\documentclass[letterpaper,11pt]{amsart}
\usepackage{tikz, tikz-cd}
\usepackage{pgfplots}
\usetikzlibrary{arrows}
\usepackage{subcaption} 
\usepackage[font=footnotesize,labelfont=bf]{caption}
\usepackage{blkarray}
\usepackage{enumitem}
\usepackage{hyperref}

\usepackage{latexsym,array,delarray,epsfig,setspace,cleveref, mathtools,amssymb,mathrsfs}
\usepackage{amsfonts,amssymb,amsmath,amsthm,pgfplots,tikz,tikz-cd,float,tkz-euclide}
\pgfplotsset{compat=1.16}

\definecolor{light}{gray}{.75}
\definecolor{med}{gray}{.5}
\definecolor{dark}{gray}{.25}

\newtheorem{theorem}{Theorem}
\numberwithin{theorem}{section}
\newtheorem{proposition}[theorem]{Proposition}

\newtheorem{corollary}[theorem]{Corollary}
\newtheorem{lemma}[theorem]{Lemma}
\newtheorem{conjecture}[theorem]{Conjecture}

\theoremstyle{definition}
\newtheorem{definition}[theorem]{Definition}
\newtheorem{remark}[theorem]{Remark}
\newtheorem{example}[theorem]{Example}

\newcommand{\C}{{\mathbb C}}

\newcommand{\Q}{{\mathbb Q}}

\newcommand{\RR}{\mathcal{R}}
\newcommand{\Z}{{\mathbb Z}}
\renewcommand{\P}{{\mathbb P}}

\newcommand{\MM}{\mathcal{M}}

\newcommand{\B}{\mathbb{B}}
\newcommand{\E}{\mathcal{E}}

\newcommand{\FL}{\mathcal{FL}}

\renewcommand{\O}{\mathcal{O}}
\newcommand{\V}{\mathcal{V}}

\renewcommand{\L}{\mathcal{L}}
\renewcommand{\S}{\textup{S}}
\newcommand{\D}{\mathcal{D}}
\newcommand{\Trop}{\textup{Trop}}
\newcommand{\T}{\mathcal{T}}
\newcommand{\I}{\mathcal{I}}

\newcommand{\be}{\mathbf{e}}

\newcommand{\bd}{\mathbf{d}}

\newcommand{\ba}{\mathbf{a}}

\newcommand{\bu}{\mathbf{u}}

\newcommand{\In}{\textup{in}}

\renewcommand{\v}{\mathfrak{v}}

\renewcommand{\In}{\textup{in}}

\newcommand{\CL}{\textup{CL}}
\newcommand{\GL}{\textup{GL}}

\newcommand{\SL}{\textup{SL}}
\newcommand{\Sym}{\textup{Sym}}
\newcommand{\Gr}{\textup{Gr}}

\newcommand{\Bpf}{\textup{Bpf}}
\newcommand{\Bl}{\textup{Bl}}

\title{Algebra and Geometry of Irreducible toric vector bundles of rank $n$ on $\P^n$}
\author{Courtney George, Christopher Manon}

\begin{document}

\maketitle

\begin{abstract}
We construct a presentation for the Cox ring of the projectivization $\P\E$ of any rank $n$ irreducible toric vector bundle on $\P^n$.  We use this presentation to show that $\P\E$ always satisfies Fujita's freeness and ampleness conjectures.
\end{abstract}

\tableofcontents

\section{Introduction}

A toric vector bundle $\pi: \E \to \P^n$ is a vector bundle equipped with the additional information of an algebraic action by the torus $T\cong (\C^*)^n$.  This action is required to intertwine with the usual action on $\P^n$, and induce linear maps between fibers. We let $\P\E$ denote the \emph{projectivization} of $\E$. Like toric varieties, projectivized toric vector bundles are a rich class of algebraic varieties which display a wide range of behavior, yet admit a combinatorial description. 

In \cite{Kaneyama88}, Kaneyama shows that any toric vector bundle on $\P^n$ of rank $<n$ must split into a direct sum of line bundles, and gives a complete classification of the irreducible toric vector bundles of rank $n$. Let $\O(a)$ for $a \in \Z$ denote the line bundle of degree $a$ over $\P^n$. Any irreducible toric vector bundle of rank $n$ can be described as the cokernel $\E_\ba$ of a map $\O \to \bigoplus_{j = 0}^n \O(a_i)$ where $\ba = \{a_0, \ldots, a_n\}\subset\Z_{> 0}$, or as the dual $\E_\ba^\vee$ of such a bundle. In particular, the tangent bundle $\T\P^n$ and cotagent bundle $\T^\vee\P^n$ are the case $a_j = 1$.  In this paper we study the projectivizations $\P\E_\ba$ and $\P\E_\ba^\vee$ of these bundles.  

\subsection{Mori dream spaces}

Recall that a complete, normal variety $X$ is said to be a \emph{Mori dream space} if its Cox ring $\RR(X)$ is finitely generated. See the book \cite{ADUH-book} for background on Cox rings.  In \cite{HMP}, Hering, Musta\c{t}\u{a}, and Payne ask when a projectivized toric vector bundle is a Mori dream space. The first work on this question is due to Hausen and S\"u\ss \ \cite{Hausen-Suss}, where they answer the question in the positive for tangent bundles of toric varieties.  In \cite{Gonzalez-rank2}, Gonzal\'ez shows that all rank $2$ vector bundles give Mori dream spaces.  Many non-examples are found by Gonzal\'ez, Hering, Payne, and S\"u\ss \ \cite{GHPS} by relating projectivized toric vector bundles to blow-ups of projective spaces. In what follows we let $\RR(\P\E)$ denote the \emph{Cox ring} of the projectivized bundle $\P\E$. 

In \cite{Kaveh-Manon-tvb}, Kaveh and the 2nd author show that the data of a toric vector bundle can be encoded in a pair $(L, D)$, where $L$ is a linear ideal, and $D$ is an integral matrix called the \emph{diagram}. Sufficient conditions for this data to define a bundle $\E$ with $\P\E$ is a Mori dream space are given \cite{Kaveh-Manon-tvb}, \cite{George-Manon}, \cite{George-Manon-Positivity}. In Section \ref{sec-diagram} we find pairs $(L_n, D_\ba)$ and $(L_n^\vee, D_\ba^\vee)$ for $\E_\ba$ and $\E_\ba^\vee$, respectively. Notably, the ideals $L_n$ and $L_n^\vee$ depend only on the dimension $n$.

\begin{theorem}\label{thm-main-pairs}
For any $\ba \subset \Z_{> 0}$ the diagram $D_\ba$ is the diagonal matrix with entries $\ba$, and the ideal $L_n \subset \C[y_0, \ldots, y_n]$ is generated by $y_0 + \cdots + y_n$.
The ideal $L_n^\vee \subset \C[z_{ij} \mid 0 \leq i < j \leq n]$ is generated by the forms $z_{ik} - z_{ij} - z_{jk}$, for all $i < j < k$, and the diagram $D_\ba^\vee$ has $i, jk$-th entry $0$ if $i \notin \{j, k\}$ and $-a_i$ otherwise.
\end{theorem}

\begin{example}
Let $\ba = \{a_0, a_1, a_2, a_3\} \subset \Z_{>0}$, and let $\E_\ba$ and $\E_\ba^\vee$ be the corresponding irreducible bundles of rank $3$ on $\P^3$. We describe the pairs $(L_3, D_\ba)$ and $(L_3^\vee, D_\ba^\vee)$ for these bundles. 

For $\E_\ba$ we have $L_3 = \langle y_0 + y_1 + y_2 + y_3\rangle \subset \C[y_0, y_1, y_2, y_3]$, and\\

\[D_\ba = \begin{bmatrix}
a_0 & 0 & 0 & 0 \\
0 & a_1 & 0 & 0 \\
0 & 0 & a_2 & 0 \\
0 & 0 & 0 & a_3 \\
\end{bmatrix}.\]\\

\noindent
For $\E_\ba^\vee$ we have $L_3^\vee = \langle z_{03} - z_{02} - z_{23}, z_{02} - z_{01} - z_{12}, z_{13} - z_{12} - z_{23}\rangle \subset \C[z_{01}, z_{02}, z_{03}, z_{12}, z_{13}, z_{23}]$, and

\[D_\ba^\vee = \begin{bmatrix}
-a_0 & -a_0 & -a_0 & 0 & 0 & 0 \\
-a_1 & 0 & 0 & -a_1 & -a_1 & 0 \\
0 & -a_2 & 0 & -a_2 & 0 & -a_2 \\
0 & 0 & -a_3 & 0 & -a_3 & -a_3 \\
\end{bmatrix}.\]\\

\end{example}

The matroid defined by the ideal $L_n^\vee$ is that of the type $A_{n+1}$ root system, or equivalently, the graphical matroid defined by a complete graph on $n+1$ vertices. Theorem \ref{thm-main-pairs} has several nice corollaries when it is paired with results from \cite{George-Manon}.  For any toric vector bundle $\E$ there is an associated \emph{full flag bundle} $\FL(\E)$, see \cite{George-Manon}, and Section \ref{sec-MDS}.

\begin{corollary}\label{cor-flag}
Let $V$ be a finite dimensional vector space, and $\ba \subset \Z_{>0}$, then $\P(\E_\ba\otimes V)$, $\P(\E_{\ba}^\vee\otimes V)$, $\FL(\E_\ba \otimes V)$, and $\FL(\E_\ba^\vee \otimes V)$ are Mori dream spaces.
\end{corollary}

\noindent
In Section \ref{sec-MDS} we give complete presentations of the Cox rings $\RR(\P\E_\ba)$, $\RR(\P\E_\ba^\vee)$ and $\RR(\FL(\E_\ba)$. By Kaneyama's result \cite{Kaneyama88}, any toric vector bundle of rank $< n$ on $\P^n$ splits, so the statement of Corollary \ref{cor-flag} actually applies to any toric vector bundle of rank $\leq n$ on $\P^n$. For an example of a rank $3$ bundle on $\P^2$ which is not a Mori dream space, see \cite[Example 6.9]{Kaveh-Manon-tvb}.

The ideals $\I_\ba$ and $\I_\ba^\vee$ which appear in the presentations of $\RR(\P\E_\ba)$ and $\RR(\P\E_\ba^\vee)$ are interesting in their own right. After reindexing, the ideal $\I_\ba^\vee$ is generated by modified \emph{Pl\"ucker relations} (\cite[Section 14.4]{Miller-Sturmfels}), see Proposition \ref{prop-inv-dual}.  In particular, a generating set for $\I_\ba^\vee$ can be obtained by taking the Pl\"ucker relations on $P_{ij}$ for $0 \leq i < j \leq n+1$, and replacing any instance of $P_{0i}$ with $P_{0i}^{a_i}$.

\begin{figure}[ht]
\begin{tikzpicture}
\node[shape=circle,draw=black] (0) at (0,-4) {0};
\node[shape=circle,draw=black] (1) at (-4,-1) {1};
\node[shape=circle,draw=black] (2) at (-2,3) {2};
\node[shape=circle,draw=black] (3) at (2,3) {3};
\node[shape=circle,draw=black] (4) at (4,-1) {4};

\path [draw = blue, line width=0.5mm] (0) edge node[left] {$a_1$} (1);
\path [draw = blue, line width=0.5mm] (0) edge node[left] {$a_2$} (2);
\path [draw = blue, line width=0.5mm] (0) edge node[right] {$a_3$} (3);
\path [draw = blue, line width=0.5mm] (0) edge node[right] {$a_4$} (4);

\path [-] (1) edge node[left] {} (2);
\path [-] (1) edge node[left] {} (3);
\path [-] (1) edge node[left] {} (4);

\path [-] (2) edge node[left] {} (3);
\path [-] (2) edge node[left] {} (4);

\path [-] (3) edge node[left] {} (4);

\end{tikzpicture}
\caption{}
\label{fig-pluckers}
\end{figure}

The notion of \emph{well-poised ideal} was introduced by the 2nd author and Ilten in \cite{Ilten-Manon} to describe favorable properties of the initial ideals of complexity one $T$-varieties.  An ideal $I$ is said to be well-poised if any initial ideal $\In_w(I)$ associated to a tropical point $w \in \Trop(I)$ is a prime ideal.  It's significantly easier to compute a Newton-Okounkov body \cite{Lazarsfeld-Mustata}, \cite{Kaveh-Khovanskii} for a variety with a well-poised embedding.  There is a Newton-Okounkov body associated to every maximal face of $\Trop(I)$ for a well-poised ideal $I$. Work of Escobar and Harada \cite{Escobar-Harada} on wall-crossing for Newton-Okounkov bodies shows that there are piecewise-linear bijections relating bodies associated to adjacent faces. The quintessential example of a well-poised ideal is the ideal generated by the Pl\"ucker relations which cut out the Grassmannian variety $\Gr_2(n+2)$.  This ideal is known to coincide with $\I_{1, \ldots, 1}$, which presents the Cox ring of the projectivized cotangent bundle $\P\T^\vee\P^n$.  The following is a natural generalization of this fact. 

\begin{theorem}\label{thm-main-wellpoised}
    The ideals $\I_\ba$ and $\I_\ba^\vee$ are well-poised.  
\end{theorem}

Using Theorem \ref{thm-main-wellpoised} we give a procedure for compute Newton-Okounkov bodies for the projectivization of any irreducible toric vector bundle of rank $n$ on $\P\E$ in Section \ref{sec-nok}.  We also show that we can control the Mori dream space property under pullback along any toric blow-up $\textup{BL}_\rho\P^n \to \P^n$ corresponding to adding a ray to the fan of $\P^n$, see Corollary \ref{cor-blowup}

\begin{remark}
    The Grassmannian variety $\Gr_2(n+2)$ is also a cluster variety.  It would be interesting to identify a connection between the bundles $\P\E_\ba^\vee$ and the theory of cluster varieties.  
\end{remark}

\subsection{Divisors and Fujita's conjectures}

For a smooth, projective variety $X$ we let $\CL(X)$ denote the divisor class group, and $K_X$ denote the canonical class of $X$.  The following are Fujita's freeness and ampleness conjectures, respectively. 

\begin{conjecture}\label{conj-Fujita}
Let $X$ be smooth of dimension $n$ and let $A \in \CL(X)$ be ample, then\\

\begin{enumerate}
    \item for $m \geq n + 1$, $K_X + m A$ is globally generated,
    \item for $m \geq n + 2$, $K_X + m A$ is very ample.\\
\end{enumerate}
\end{conjecture}

\noindent
By Mori's cone theorem (\cite[Theorem 1.5.33]{Lazarsfeld}), Conjecture \ref{conj-Fujita} can be proven for a variety $X$ by showing that any Nef class on $X$ is basepoint-free, and any ample class on $X$ is very ample. Techniques for computing positivity properties of divisors on a projectived toric vector bundle $\P\E$ are developed in \cite{George-Manon-Positivity}.  In Section \ref{sec-positivity} we apply modifications of those techniques to compute the pseudo-effective and Nef cones, along with the effective and semiample monoids, for any $\P\E_\ba$ and $\P\E_\ba^\vee$.  We then deduce the following. 

\begin{theorem}\label{thm-main-positivity}
For any $\ba \subset \Z_{> 0}$, a divisor on $\P\E_\ba$ or $\P\E_\ba^\vee$ is basepoint free if it is Nef, and very ample if it is ample. As a consequence, $\P\E_\ba$ and $\P\E_\ba^\vee$ satisfy Fujita's freeness and ampleness conjectures.  
\end{theorem}

The Fujita conjectures are known for smooth toric varieties.  In the non-smooth toric case, freeness holds by work of Fujino \cite{Fujino}, and a theorem of Payne \cite{Payne-Ampleness} establishes ampleness in the case of Gorenstein singularities. Working by way of dimension, curves satisfy both conjectures by Riemann-Roch, and results of Reider \cite{Reider} prove the surface case. The freeness conjecture has been proved for smooth projective varieties up to dimension $5$ \cite{Ein-Lazarsfeld,Kawamata,Ye-Zhu}.  For toric vector bundles, if the rank of $\E$ is $2$, then $\P\E$ is a \emph{complexity-$1$ $T$-variety}, and therefore satisfies Fujita's freeness conjecture by a result of Altmann and Ilten \cite{AltmannIlten}. We credit Altmann and Ilten for the suggestion to study the Fujita conjectures for projectivized toric vector bundles. For Mori dream spaces, Fahrner \cite{Fahrner} has developed algorithms which test the freeness conjecture. 

We prove Theorem \ref{thm-main-positivity} by showing that $\P\E_\ba$ and $\P\E_\ba^\vee$ both carry embeddings into split projectivized toric vector bundles in such a way that both the class groups and semiample monoids are isomorphic to those of the ambient space (Proposition \ref{prop-neatandtidy}).  We call a map with these properties a \emph{neat and tidy} embedding, see Section \ref{sec-positivity}. 

\bigskip

\noindent{\bf Acknowledgements:} We thank Kiumars Kaveh for many useful discussions about toric vector bundle. 

\section{Linear ideals and diagrams}\label{sec-diagram}

In what follows let $E$ denote the fiber of a toric vector bundle $\E$ over the identity point of a torus $T$, thought of as a dense, open subvariety of a smooth toric variety $X(\Sigma)$. The data of a toric vector bundle $\E$ can be packaged in a number of ways.  Kaneyama's classification \cite{Kaneyama} is by certain $\GL(E)$ cocycles. Klyachko \cite{Klyachko} uses an arrangement of filtrations of $E$ labelled by data from the representation theory of $T$.  In \cite{Kaveh-Manon-tvbval}, Kaveh and the $2$nd author show that $\E$ can be captured in a \emph{prevaluation} $\v: E \to \O_{|\Sigma|}$ where the latter denotes the integral piecewise-linear functions with a finite number of domains of linearity on the support of the fan $\Sigma$. The operations on $\O_{|\Sigma|}$ are ``multiplication," computed as pointwise sum, and ``addition," computed as pointwise minimum. We regard $\infty$ as the additive identity of $\O_{|\Sigma|}$, and $0$ serves as the multiplicative identity.  Under these operations, $\O_{|\Sigma|}$ has the structure of a \emph{semifield}.

\begin{definition}
    A prevaluation $\v:E \to \O_{|\Sigma|}$ is a function which satisfies:\\

    \begin{enumerate}
        \item $\v(0) = \infty$,
        \item $\v(Cf) = \v(f)$ for any $C \in \C\setminus \{0\}$,
        \item $\v(f + g) \geq \min\{\v(f), \v(g)\}$.\\
    \end{enumerate}
\end{definition}

Klyachko's compatibility conditions, and Kaneyama's cocycle data can be wrapped up in the axioms of prevaluations, along with the requirement that there be an \emph{linear adapted basis} $\B_\sigma$ for the restriction $\v\!\!\mid_\sigma: E \to \O_{|\sigma|}$ of $\v$ to each maximal face of the fan $\Sigma$.  This means that there is some vector space basis $\B_\sigma \subset E$ with the property that $\v\!\!\mid_\sigma(b)$ agrees with an integral linear form on $|\sigma|$, and for any $f \in E$ with $f = \sum_{b_i \in \B_\sigma} c_ib_i$ we have $\v \mid_\sigma(f) = \min\{\v(b_i) \mid c_i \neq 0\}$. 

Because of the adapted bases $\B_\sigma \subset E$ and the smoothness of $\Sigma$, the prevaluation $\v$ is entirely determined by its values on the ray generators of $\Sigma$. The fact that $\v$ is a prevaluation implies that its specialization $\v(u): E \to \Z$ at a ray generator determines an integral, decreasing filtration of $E$ by the spaces:\\

\[F_r^u = \{ f \mid \v(u)(f) \geq r\}.\]\\

\noindent
The integral filtration $F^u$ of $E$ can be used to define an integral valuation on the polynomial ring $\Sym(E)$.  A choice of spanning set $\mathcal{B} = \{b_1, \ldots, b_m\} \subset E$ determines a presentation of $\Sym(E)$ by a linear ideal $L \subset \C[y_1, \ldots, y_m]$, and by taking values on $\mathcal{B} \subset E$, the valuation defined by $F^u$ specializes to a point $w(u)$ in the tropical variety $\Trop(L)$ by evaluation on the image of the generators $y_1, \ldots, y_m$ in $\Sym(E)$.  The point $w(u)$ completely determines $F^u$ in that $F^u_r$ is the span of those elements of $\mathcal{B}$ whose $w(u)$ entries are larger than $r$. Such a set always defines a flat of the matroid $\MM(L)$ determined by $L$ on the set $y_1, \ldots, y_m$. 

If we take $\mathcal{B}$ to be any spanning set of $E$ containing adapted bases $\B_\sigma$ we arrive at the characterization of toric vector bundles in \cite[Theorem 1.4]{Kaveh-Manon-tvb}. A toric vector bundle can be captured by a configuration $w_1, \ldots, w_n$ of points on a tropicalized linear space $\Trop(L)$ such that for any face $\sigma \in \Sigma$ the rows $w_i$ for $\rho_i \in \sigma(1)$ must all belong to a common \emph{apartment} $A_\B \subset \Trop(L)$.  Apartments $A_\B \subset \Trop(L)$ are distinguished polyhedral subcomplexes of $\Trop(L)$. There is an apartment $A_\B$ for each basis $\B \subset \MM(L)$, in particular, $A_\B$ is the set of $(v_1, \ldots, v_m) \in \Trop(L)$, where $v_j$ is equal to the minimum of the $v_k$ where $y_k$ appears in the $\B$-expression for $y_j$. It is straightforward to show that the apartments cover $\Trop(L)$ and that each $A_\B$ is piecewise-linear isomorphic to $\Q^r$.  We organize the configuration into an $n \times m$ matrix $D$, where the $i$-th row of $D$ is $w_i$. In this way, a toric vector bundle $\E$ is determined by a pair $(L, D)$. 

We let $\Sigma_n$ denote the fan of $\P^n$. The ray generators of $\Sigma_n$ are the elementary basis vectors $\be_1, \ldots, \be_n$, along with $\be_0 = -\sum_{i =1}^n \be_i$. We let $\sigma_i$ denote the maximal face of $\Sigma_n$ which is spanned by all ray generators except $\be_i$.  Recall that a $T$-linearized line bundle $\O(r_0, \ldots, r_n)$ on $\P^n$ is determined by a tuple $(r_0, \ldots, r_n) \in \Z^{n+1}$.  There is a corresponding piecewise-linear function $\phi$ on $\Sigma_n$ determined by the properties that $\phi\!\!\mid_{\sigma_i}$ is linear, and $\phi(\be_i) = r_i$. 

Tensoring a bundle $\E$ by a line bundle $\L$ does not change the projectivization, so we may replace the diagram of $\E$ by the diagram of $\E\otimes \L$ without changing the geometry of $\P\E$.  Tensoring a toric vector bundle $\E$ which corresponds to the pair $(L, D)$ by the $T$-linearized line bundle $\O(r_0, \ldots, r_n)$ amounts to adding $r_j$ to each entry of the $j$-th row of the diagram $D$. For this reason we may speak of the \emph{non-negative} form of a $(L, D)$, where we have added non-negative integers to the row of $D$ so that each entry is non-negative. 

\subsection{The pair $(L_n, D_\ba)$}

Fix $\ba = \{a_0, \ldots, a_n\} \subset \Z_{> 0}$.  In \cite{Kaneyama88}, Kaneyama constructs half of the irreducible bundles of rank $n$ on $\P^n$ as cokernels of the maps $\O \to \bigoplus_{j = 0}^n \O(a_j)$.  We partially reconstruct Kaneyama's result by supposing we have an exact sequence of toric vector bundles:\\

\[ 0 \to \O(0) \to \bigoplus_{i =0}^n \O(D_i) \to \E \to 0,\]\\

\noindent
with $\E$ irreducible and rank $n$, where each $D_j \in \Z^{n+1}$ satisfies $\sum_{i = 0}^n D_{ij} = a_j$, and determines an action of $T$ on the total space of $\O(a_j)$. We let $(E, \v)$ be the prevalued vector space corresponding to $\E$. 

Now we let $\V_D = \bigoplus_{j = 0}^n \O(D_j)$ and $V = \bigoplus_{j = 0}^n \C y_j$. Each tuple $D_j$ determines the piecewise-linear function $\phi_j \in \O_{\P^n}$, where  $\phi_j(\be_i) = D_{ij}$.  We let $\v_D: V \to \O_{\P^n}$ be the prevaluation with adapted basis $\{y_0, \ldots, y_n\}$ such that $\v_D(y_j) = \phi_j$.  By \cite[Section 3]{Kaveh-Manon-tvb}, this corresponds to a map of prevalued vector spaces $(V, \v_D) \to (E, \v)$.  The surjectivity of $\pi$ implies in particular that the induced maps $\pi_\sigma: (\V_D)_\sigma \to \E_\sigma$ over the torus fixed points of $X(\Sigma)$ are also surjective. Consequently, the images $b_0, \ldots, b_n$ of the $y_0, \ldots, y_n \in V$ contains an adapted basis for $\v\!\!\mid_{\sigma_i}$, for each maximal face of $\Sigma_n$.  Accordingly, we may take the spanning set $\mathcal{B} = \{b_0, \ldots, b_n\} \subset E$.  As $E$ has dimension $n$, precisely one linear relation can hold among the $b_j$. Without loss of generality, we take this to be $b_0 + \ldots + b_n = 0$. We conclude that $L = \langle y_0 + \cdots + y_n\rangle$, and $D = [D_0 \cdots D_n]$ for the bundle $\E$. The following proposition finishes our analysis. 

\begin{proposition}\label{prop-pair-ba}
The bundle $\E_\ba$ is comes from the pair $(L_n, D_\ba)$, where $L_n = \langle y_0 + \cdots + y_n\rangle$ and $D_\ba$ is the diagonal matrix with the $a_j$ along the diagonal. 
\end{proposition}

\begin{proof}
First, we observe that each collection of $n$ rows of $D_\ba$ must share a common adapted basis.  This implies that for each such collection, there is a column with all $0$ entries in the associated rows; namely the column corresponding to the element left out of the basis.  Next, observe that if any two collections have the same adapted basis, then all rows have the same adapted basis.  This would imply that the resulting bundle would be split, contradicting the irreducibility of $\E_\ba$. It follows that each of the $n+1$ collections of rows have a distinct adapted basis.  There are exactly $n+1$ bases of the matroid defined by $L_n$, and each element of $y_0, \ldots, y_n$ is left out of some basis.  Accordingly, each column of $D_\ba$ has at most one non-zero entry in the location not in the corresponding complementary set of rows, and irreducibility again implies that this entry is positive.  We can therefore assume that $D_\ba$ is a diagonal matrix with the $a_j$ along the diagonal.
\end{proof}

\subsection{The pair $(L_n^\vee, D_\ba^\vee)$}

Let $(E, \v)$ be a prevalued vector space, and let $\E$ be the corresponding toric vector bundle over a toric variety $X(\Sigma)$. The dual bundle $\E^\vee$ corresponds to a prevaluation $\v^\vee: E^\vee \to \O_{|\Sigma|}$ on the dual $E^\vee$ of $E$. The function $\v^\vee$ is determined by what it does on the duals of the adapted bases $\B_\sigma^\vee \subset E^\vee$ for $\sigma \in \Sigma$. Following \cite[Example 3.27]{Kaveh-Manon-tvb}, we take $\B_\sigma^\vee$ to be the dual basis of $\B_\sigma$, and we have $\v^\vee\mid_\sigma(b_i^\vee) = -\v\!\!\mid_\sigma(b_i)$.  

Now we wish to go from the data $(E^\vee, \v^\vee)$ to a pair $(L^\vee, D^\vee)$ for a dual bundle.  The ideal is obtained by taking $L^\vee$ to be the relations which vanish on the set $\mathcal{B}^\vee = \cup_{\sigma \in \Sigma} \B_\sigma^\vee$.  Now fix a face $\sigma \in \Sigma$, and consider the rows of $D^\vee$ which come from the rays $\sigma(1)$.  The entries of these rows coming from $\B_\sigma^\vee$ are determined by the formula $\v^\vee\!\!\mid_\sigma(b_i^\vee) = -\v\!\!\mid_\sigma(b_i)$, and any value for $b \in \mathcal{B}^\vee \setminus \B_\sigma^\vee$ is computed by expressing $b$ as a unique linear combination of elements of $\B_\sigma^\vee$, and taking the minimum value which appears. This method determines all entries of $D^\vee$.  

We now apply this construction to $\E_\ba^\vee$. Let $E = V/\langle y_0 + \cdots y_n\rangle$. We view $E^\vee$ as the space of linear functionals $z: V \to \C$ such that $\sum_{j = 0}^n z(y_j) = 0$.  For any basis member $y_j \in V$ we have a dual $y_j^\vee: V \to \C$, and the difference $z_{ij} = y_i^\vee - y_j^\vee$ of any two dual basis members is always an element of $E^\vee$.  

Fix a face $\sigma_i$ in $\Sigma_n$, then $\B_{\sigma_i}$ has adapted basis the images of the $y_j$ with $j \neq i$. It is then easy to check that $\B_{\sigma_i}^\vee = \{z_{ji} \mid j \neq i\} \subset E$. Moreover, $\v^\vee_\ba(z_{ji}) = -\v_\ba(y_j)$, this vector is $0$ for each ray in $\sigma_i$ except the $j$-th ray, where it is $-a_j$. To simplify matters we take $\mathcal{B}^\vee = \{z_{ji} \mid j < i \} \subset E^\vee$; then this set contains an adapted basis for each $\sigma_i$.

\begin{proposition}\label{prop-pair-dual-ba}
The ideal $L_n^\vee$ is generated by the circuits $z_{ik} = z_{ij} + z_{jk}$ for $i < j < k$.  The non-negative diagram $D_\ba^\vee$ has rows $w_\ell$ with the property that the $ij$-th entry is $-a_\ell$ if $\ell = i$ or $j$ and $0$ otherwise.   
\end{proposition}

\begin{proof}
    Fix $i < j < k$, then $z_{ij} + z_{jk} = y_i^\vee - y_j^\vee + y_j^\vee - y_k^\vee = z_{ik}$. These are the relations of the type $A_n$ root system, so they define a quotient space of rank $n$, and generate $L_n^\vee$.  Now fix a row $w_\ell$ of $D_\ba^\vee$.  The $\ell$-th ray appears in the face $\sigma_j$, where $j \neq \ell$, which has adapted basis $z_{ij}$.  If $i = \ell$, then the $\ell, ij$-th entry is $-a_\ell$; otherwise this entry is $0$. This determines $D_\ba^\vee$. 
\end{proof}

We have determined the pair $(L_n^\vee, D_\ba^\vee)$ of the dual bundle $\E_\ba^\vee$, however the diagram $D_\ba^\vee$ has negative entries.  The non-negative form of this pair, which best describes $\P\E_\ba^\vee$ has diagram $D_\ba'$, where the $\ell, ij$-th entry is $0$ if $\ell = i$ or $j$ and $a_\ell$ otherwise. 

\section{Cox Rings}\label{sec-MDS}

In this section we find presentations of the Cox rings $\RR(\P\E_\ba)$, $\RR(\P\E_\ba^\vee)$ of the projectivizations of the irreducible toric vector bundles of rank $n$ on $\P^n$.  First we use a result in \cite{Kaveh-Manon-tvb} to show that $\RR(\P\E_\ba)$ is presented by a single hypersurface.  The matrix $\D_\ba$ is diagonal, so we are able to use a theorem from \cite{George-Manon} to conclude that the full flag bundle $\FL(\E_\ba)$ is also a Mori dream space.  We find a presentation of the Cox ring $\RR(\FL(\E_\ba))$ by adapting an argument from \cite{George-Manon}.  We then use an invariant theory construction to find a presentation of $\RR(\P\E_\ba^\vee)$.  The ideals $\I_\ba$, $\I_\ba^\vee$ of these presentations are then shown to be well-poised. 

\subsection{The Cox ring $\RR(\P\E_\ba)$}

The irreducible bundle $\E_\ba$ is an example of a \emph{sparse} bundle (\cite{Kaveh-Manon-tvb}, \cite{George-Manon}.  In particular, the diagram $D_\ba^\vee$ has at most one non-zero entry in each row.  The Klyachko filtrations of sparse bundles have at most one intermediary step, which are also codimension $1$. Sparse bundles were first studied in \cite{GHPS}, where it was shown they are Mori dream spaces, and explicit presentations of their Cox rings were constructed.  The following presentation comes from \cite[Corollary 6.7]{Kaveh-Manon-tvb}, which states that the Cox ring of a sparse pair $(L, D)$ is presented by appropriate homogenizations of a minimal generating set of $L$. In particular, the Cox ring of the projectivization of a sparse bundle is always a complete intersection, and any sparse bundle is a \emph{CI bundle}, see \cite[Proposition 6.2]{Kaveh-Manon-tvb}.

\begin{proposition}\label{prop-presentation-ba}
The Cox ring $\RR(\P\E_\ba)$ is the quotient of the polynomial ring in $2(n+1)$ variables $\C[x_0, \ldots, x_n; Y_0, \ldots, Y_n]$ by the ideal $\I_\ba = \langle \sum_{j =0}^n x_j^{a_j}Y_j \rangle$.    
\end{proposition}

\begin{example}
    The tangent bundle $\T\P^n$ is the case $\ba = (1, \ldots, 1)$, we recover the well-known fact that the Cox ring of $\P\T\P^n$ is presented by the polynomial $x_0Y_0 + \cdots + x_nY_n$. 
\end{example}

\subsection{The full flag bundles $\FL(\E_\ba)$ and $\FL(\E_\ba^\vee)$}

The next proposition is largely the content of \cite{George-Manon}, although we have added a treatement of the dual bundle $\E^\vee$. If the full flag bundle $\FL(\E)$ is a Mori dream space, an infinite collection of projectivized toric vector bundles and their full flag bundles must also be Mori dream spaces. The methods used to prove this result draw on the theory of \emph{representation stability} and non-reductive representation theory, see 
 \cite{George-Manon} for details.  

\begin{proposition}\label{prop-flagdual}
    Let $\E$ be a toric vector bundle, and suppose $\FL(\E)$ is a Mori dream space, then for any finite-dimensional vector space $V$, $\P(\E \otimes V)$, $\FL(\E\otimes V)$, $\FL(\E^\vee \otimes V)$, and $\P(\E^\vee \otimes V)$ are Mori dream spaces. 
\end{proposition}

\begin{proof}
The fact that $\P(\E\otimes V)$ and $\FL(\E\otimes V)$ are Mori dream spaces is \cite[Theorem 1.3, Corollary 1.4]{George-Manon}.  If $\FL(\E\otimes V)$ is a Mori dream space, then $\P((\E\otimes V)^\vee) = \P(\E^\vee\otimes V^\vee)$ is a Mori dream space for all $V$ by \cite[Corollary 5.9]{George-Manon}, see also the discussion in Section \ref{subsec-Coxdual}. But this implies that $\FL(\E^\vee \otimes V)$ is a Mori dream space for all $V$.
\end{proof}

\begin{corollary}\label{cor-allirredflag}
Let $\E$ be an irreducible toric vector bundle of rank $n$ on $\P^n$, then $\P(\E\otimes V)$ and $\FL(\E\otimes V)$ are Mori dream spaces.     
\end{corollary}

\begin{proof}
By Proposition \ref{prop-flagdual}, it suffices to show that $\FL(\E_\ba)$ is a Mori dream space, but this is \cite[Corollary 5.8]{George-Manon}.
\end{proof}

Now we give a presentation of the Cox ring $\RR(\FL(\E_\ba))$. Our treatment follows the construction of the Cox ring of $\FL(\T\P^n)$ in \cite{George-Manon} (this is the special case $\ba = (1, \ldots, 1)$).  For toric vector bundle $\E$ of rank $n$, the Cox ring $\RR(\FL(\E))$ can be constructed as the algebra of invariants by the natural action of the unipotent group $U_{n-1}$ of $n-1\times n-1$ upper triangular matrices with $1$'s on the diagonal on the Cox ring $\RR(\P(\E\otimes \C^{n-1}))$.  In the case $\E_\ba$, the algebra $\RR(\P(\E_\ba \otimes \C^{n-1}))$ is presented by the polynomial ring $\C[x_j, Y_{ij} \mid 0 \leq j \leq n, 1 \leq i \leq n-1]$, which also carries an action by $U^{n-1}$.  The unipotent actions extend to an action by the reductive group $\GL_{n-1}(\C)$, hence we obtain a surjection:\\

\[\C[x_j, Y_{ij} \mid 0 \leq j \leq n, 1 \leq i \leq n-1]^{U_{n-1}} \to \RR(\FL(\E_\ba)) \to 0.\]\\

\noindent
The algebra $\C[x_j, Y_{ij} \mid 0 \leq j \leq n, 1 \leq i \leq n-1]^{U_{n-1}} \subset \C[x_j, Y_{ij} \mid 0 \leq j \leq n, 1 \leq i \leq n-1]$ is a polynomial ring in $n+1$ variables over the Pl\"ucker algebra of upper-justified minors of the matrix of variables $[Y_{ij}]$.  For background on the Pl\"ucker relations see \cite[Section 14.2]{Miller-Sturmfels}.  Accordingly, we can present $\RR(\FL(\E_\ba))$ as a quotient of the polynomial ring $\C[x_j, P_{\tau}, P_{0, \tau} \mid 0 \leq j \leq n, \tau \subset [n]]$. 

The Cox ring $\RR(\FL(\E_\ba))$ is the image of the map $\Psi_\ba: \C[x_j, P_\tau, P_{0,\tau}] \to \C[t_j, y_{ij}]$ defined by:\\

\[\Psi_\ba(x_j) = t_j^{-1},\]
\[\Psi_\ba(P_\tau) = \det[y(\tau)]t^{\ba_\tau},\]
\[\Psi_\ba(P_{0,\tau}) = \sum_{j =1}^n \det[y(j,\tau)]t_0^{\ba_0}t^{\ba_\tau}.\]\\

\noindent
Here $y(\tau)$ denotes the minor on the first $|\tau|$ rows and the $\tau$ columns of $[y_{ij}]$, and $\ba_\tau$ denotes the part of the tuple $\ba$ supported on $\tau$. 
 The quadratic Pl\"ucker relations hold among the $P_\tau$ and $P_{0,\tau}$. Moreover, it is straightforward to very that $\sum_{j \notin \tau} x_j^{\ba_j}P_{j \tau} =0$ using the definition of $\Psi_\ba$.  We show that these relations suffice to present $\RR(\FL(\E_\ba))$.

The proof of Theorem \ref{thm-fullflag} involves a modification of the semigroup $GZ_n$ of Gel'fand-Zetlin patterns with $n$ rows. A Gel'fand-Zetlin pattern $g \in GZ_n$ is an  array of integers arranged in $n$ rows, where the $i$-th row has $n + 1 - i$ entries $g_{ij}$.  These integers satisfy additional inequalities: $g_{ij} \geq g_{i+1,j} \geq g_{i,j+1}$.  Let $GZ_n^+$ be the set of patterns with $g_{1n} = 0$.  The generators of $GZ_n^+$ are in bijection with strict subsets $\tau \subset [n]$, where the pattern $g(\tau)$ corresponding to $\tau$ is the unique pattern with $|\tau \cap [n-i + 1]|$ $1$'s and $|[n-i + 1] \setminus \tau|$ $0$'s in row $i$. 

A solution to the \emph{word problem} for a semigroup $S$ with generators $A_1, \ldots, A_m$ is a minimal set of relations of the form $A_{i_1}\cdots A_{i_t} = A_{j_1}\cdots A_{j_s}$ which can be used to transform an arbitrary relation to a trivial relation by substitution.

\begin{theorem}\label{thm-fullflag}
The ideal $\ker(\Psi_\ba)$ is generated by the Pl\"ucker relations among the $P_\tau$ and $P_{0,\tau}$, along with relations of the form $\sum_{j \notin \tau} x_j^{\ba_j}P_{j \tau} =0$.  
\end{theorem}

\begin{proof}
The argument proceeds as in \cite[Algorithm 1.8]{Kaveh-Manon-NOK}.  We select a monomial ordering on $\C[t_j, y_{ij}]$ which satisfies $y_{ij} \prec t_\ell$ for all $i, j, \ell$, such that the initial form $\In_\prec\det[y(\tau)]$ is the product of the diagonal terms. This ordering also defines a partial ordering on the variables $x_j, P_\tau, P_{0, \tau}$. The initial forms of $\Psi_\ba(x_j), \Psi_\ba(P_\tau),$ and $\Psi_\ba(P_{0,\tau})$ with respect to this ordering generate a semigroup in $GZ_n\times \Z^{n+1}$. To prove the theorem, it suffices to show that there is a generating set of the binomial relations which vanish on these initial forms which are themselves the initial forms of the described set of relations with respect to the induced partial ordering on the polynomial ring $\C[x_j, P_\tau, P_{0, \tau}]$. 

Now we identify the initial forms $\In_\prec t_j^{-1}$, $\In_\prec \det[y(\tau)]t^{\ba_\tau}$, and $\In_\prec \det[y(0,\tau)]t_0^{\ba_0}t^{\ba_\tau}$ with elements of the semigroup $GZ_n\times\Z^{n+1}$.  The form $\In_\prec t_j^{-1}$ is sent to $(0, -\be_j) \in GZ_n\times \Z^{n+1}$, the form $\In_\prec \det[y(\tau)]t^{\ba_\tau}$ is sent to $(g(\tau), \sum_{j \in \tau} a_j\be_j) \in GZ_n\times \Z^{n+1}$, and finally, the form $\In_\prec \sum_{j =1}^n \det[y(j,\tau)]t_0^{\ba_0}t^{\ba_\tau}$ is sent to $(g(\tau^*), a_0\be_0 + \sum_{j \in \tau} a_j\be_j) \in GZ_n\times \Z^{n+1}$.  Here $\tau^*$ denotes the set $\tau \cup \{\ell\}$, where $\ell$ is the first element of $[n]$ not in $\tau$. In particular, if $\be_0$ appears in the support of $\bu \in \Z^{n+1}$, where $(g(\eta), \bu)$ is a generator, then the first index in the support of $\eta$ is smaller than the first non-zero index of $\bu$. 

Next, we find a generating set of the prime binomial ideal which vanishes on these patterns by solving the word problem.  If $[b] \subset \eta$ then we have:\\

\[[0, -a_b\be_b][g(\eta),\sum_{j \in \eta} a_j\be_j ] = [0,-a_0\be_0][g(\eta), a_0\be_0 + \sum_{j \in \eta \setminus \{b\}} a_j\be_j ]\]\\

Here $[g(\eta), \sum_{j \in \eta} a_j\be_j]$ is the representative of $\In_\prec \det[y(\eta)]t^{\ba_\eta}$, and $[g(\eta), a_0\be_0 + \sum_{j \in \eta \setminus \{b\}} a_j\be_j ]$ is the representative of  $\In_\prec \sum_{j =1}^n \det[y(j,\eta\setminus\{b\})]t_0^{\ba_0}t^{\ba_{\eta\setminus\{b\}}}$.  We also have the standard binomial relations among the Gel'fand-Zetlin generators:\\

\[[g(\tau), \sum_{j \in \tau} a_j\be_j][g(\eta), \sum_{i \in \eta} a_i\be_i] = [g(\tau \cup \eta), \sum_{k \in \tau \cup \eta} a_k\be_k][g(\tau \cap \eta), \sum_{\ell \in \tau\cap \eta} a_\ell\be_\ell].\]\\

\noindent
These binomials hold among the initial forms of $\In_\prec \det[y(\tau)]t^{\ba_\tau}$; it is straightforward to verify that the versions involving $\In_\prec \sum_{j =1}^n \det[y(j,\tau)]t_0^{\ba_0}t^{\ba_\tau}$ work in the same way. 

The two families of binomials described lift to $\sum_{j \notin \tau} x_j^{\ba_j} P_\tau = 0$, and Pl\"ucker relations, respectively. Therefore, if we check that these relations suffice to generate the binomial ideal which vanishes on the initial forms, we have shown that the required relations generate $\ker(\Psi_\ba)$. We suppose we have two words $A_1\cdots A_n$, $B_1 \cdots B_n$ whose product maps to the same extended Gel'fand-Zetlin pattern. We must show that after applications of the above binomial relations, these words can be taken to a common word. 

First we observe that the contributions to the $\Z^{n+1}$ component from either word agree, and that if $\be_\ell$ appears with multiplicity $b_\ell$, we can divide $b_\ell$ by $a_\ell$ and factor the resulting number of copies of $[0, -\be_\ell]$ off both words. In particular, we may assume without loss of generality that the contribution of $\be_\ell$ is divisible by $a_\ell$. Similarly, any $[0, -\be_\ell]$ not supported by a pattern elsewhere in the words can be read off of the $\Z^{n+1}$ component, and factored off both sides. Also, any $[0, -a_\ell\be_\ell]$ for $[\ell] \subset \tau$ for $[g(\tau), \sum_{j \in \tau}a_j\be_j]$ appearing the word can be converted to $[0, -a_0\be_0]$ using the first relation above. So we may assume without loss of generality that both words do not contain a generator of the form $[0, -a_\ell\be_\ell]$ for any $\ell \in \{0, \ldots, n\}$.

Now, using the Gel'fand-Zetlin binomial relations, we can ensure that the underlying Gel'fand-Zetlin patterns in both words are the same, with possibly different $\Z^{n+1}$ components. Select a leading pattern on both sides, say $A_1 = [g(\tau), \sum_{j \in \tau \setminus\{b_1\}} a_j\be_j]$ and $B_1 = [g(\tau), \sum_{j \in \tau \setminus\{c_1\}} a_j\be_j]$, and suppose that $b_1 < c_1$.  We must have that $[b_1] \subset \tau$ and $[c_1] \subset \tau$. Moreover, the $\Z^{n+1}$ components of both words agree, so there must be some $A_2 = [g(\eta), \sum_{k \in \eta \setminus \{c_1\}} a_k \be_k]$ with $[c_1] \subset \eta$.  This means $b_1, c_1 \in \tau$, and $c_1 \in [c_1] \subset \eta$, so also $b_1 \in \eta$. Now we may apply the consequence of the Gel'fand-Zetlin relations: $[g(\tau), \sum_{j \in \tau \setminus\{b_1\}} a_j\be_j][g(\eta), \sum_{k \in \eta \setminus \{c_1\}} a_k \be_k] = [g(\tau), \sum_{j \in \tau \setminus\{c_1\}} a_j\be_j][g(\eta), \sum_{k \in \eta \setminus \{b_1\}} a_k \be_k]$, and factor off $[g(\tau), \sum_{j \in \tau \setminus\{c_1\}} a_j\be_j]$. This completes the proof. 
\end{proof}

\subsection{The Cox ring $\RR(\P(\E_\ba^\vee))$}\label{subsec-Coxdual}

In this section we obtain a presentation the Cox ring $\RR(\P\E_\ba^\vee)$ from a presentation of the Cox ring $\RR(\FL(\E_\ba))$.  For a general toric vector bundle $\E$ of rank $r$ over a smooth, projective toric variety $X(\Sigma)$ the class group of $\FL(\E)$ is naturally a product:\\

\[\CL(\FL(\E)) \cong \CL(X(\Sigma))\times \CL(\FL_r),\]\\

\noindent
where $\FL_r$ denotes the full flag variety of the vector space $\C^r$, see \cite[Proposition 3.5]{George-Manon}. For a moment we draw from representation theory and write $\CL(\FL_r)$ as $\bigoplus_{i = 1}^{r-1} \Z\omega_i$, wnere $\omega_i$ denotes the $i$-th fundamental weight of the reductive group $\SL_r$. Each positive combination $\lambda = \sum_i^{r-1} n_i \omega_i$ corresponds to a Schur functor $\S_\lambda$, which is an operation which applies to both vector spaces and vector bundles. These are precisely the effective classes on $\FL_r$. If $(\bd, \lambda) \in \CL(\FL(\E))$ is effective, restriction of this class to the identity fiber of $\FL(\E)$ gives an effective class of $\FL_r$.  The Picard variety of $\FL_r$ is a point, so this implies that $\lambda$ is a positive combination of the $\omega_i$.  In this case, the global sections of $(\bd, \lambda)$ can be realizes as the following section space on $X(\Sigma)$:\\

\[H^0(\FL(\E), \O(\bd, \lambda)) \cong H^0(X(\Sigma), \O(\bd)\otimes \S_\lambda(\E)).\]\\

\noindent
In the special case that $\lambda = \sum_{i =1}^r n_i\omega_i$ with $n_i = 0$ for $i < r-1$, $\S_\lambda(\E)$ is of the form $\Sym^n(\bigwedge^{r-1}\E)$.  In this way we see that the Cox ring of $\P(\bigwedge^{r-1}\E)$ is naturally the subalgebra of $\RR(\FL(\E))$ supported on the subgroup $\CL(X(\Sigma))\times \Z\omega_{r-1} \subset \CL(\FL(\E))$. Finally, $\bigwedge^{r-1}\E$ and $\E^\vee$ differ by tensoring by the determinant line bundle $\det^{-1}(\E)$, so their projectivizations agree.  We summarize these observations below.

\begin{proposition}\label{prop-inv-dual}
    Under the grading by $\CL(\FL(\E))$, $\RR(\P(\E^\vee))$ is the subalgebra of $\RR(\FL(\E))$ supported on the subgroup $\CL(X(\Sigma))\times \Z\omega_{r-1} \subset \CL(\FL(\E))$.
\end{proposition}

Now we apply a little bit of invariant theory to find a presentation of $\RR(\P(\E_\ba^\vee))$.

\begin{proposition}\label{prop-presentation-ba-dual}
The Cox ring $\RR(\P\E_\ba^\vee)$ is the quotient of the polynomial ring $\C[x_i, Z_{jk}\mid 0 \leq i, j < k \leq n]$ by the ideal $\I_\ba^\vee$ generated by the relations $x_j^{a_j}Z_{ik} - x_k^{a_k}Z_{ij} - x_i^{a_i}Z_{jk}$ for $i < j < k$, and the quadratic Pl\"ucker relations on the $Z_{ij}$.
\end{proposition}

\begin{proof}
By Proposition \ref{prop-inv-dual}, $\RR(\P(\E_\ba^\vee))$ is the subalgebra of $\RR(\FL( \E_\ba))$ generated by the $x_i$, along with the top Pl\"ucker generators $P_{\tau}$, where $|\tau| = n-1$. We let $Z_{ij}$ denote the Pl\"ucker generator corresponding to the complement of $\{i, j\}$ in $\{0, \ldots, n\}$.  The $\CL(\FL_n)$ components of the $\CL(\FL(\E_\ba))$ degrees of a Pl\"ucker generator $P_\tau \in\RR(\FL(\E_\ba))$ is the fundamental weight $\omega_k$, where $k = |\tau|$.  Consequently, any relation among the $Z_{ij}$ must come from the $\omega_{n-1}$ component of the ideal, which is in turn generated by the Pl\"ucker relations on the $Z_{ij}$.  
\end{proof}

\begin{remark}
In the case $a_j = 1$, $\E_\ba$ is the tangent bundle of $\P^n$, making $\E_\ba^\vee$ the cotangent bundle $\T^\vee\P^n$.  The presentation above shows that the transformation $x_i \to P_{0,i+1}$, $Z_{ij} \to P_{i+1,j+1}$ takes the Cox ring $\RR(\P\T^\vee\P^n)$ isomorphically onto the Pl\"ucker algebra of the Grassmannian variety $\Gr_2(n+2)$.  See \cite{GHPS} for a different account of this isomorphism.  
\end{remark}

\subsection{The ideals $\I_\ba$ and $\I_\ba^\vee$ are well-poised}

Well-poised ideals were introduced in \cite{Ilten-Manon}, where it is shown that any rational, complexity-1 $T$-variety has a well-poised embedding.  Operations which preserve the well-poised property, in particular GIT quotients and the construction of $T$-varieties from certain polyhedral divisors, were studied in \cite{Cummings-Manon}.  The following is the main result of \cite{CDMRV}, which gives a classification of well-poised hypersurfaces. 

\begin{theorem}\label{thm-well-poised-hypersurface}
    Let $p = \sum_{i =1}^m C_{\ba_i}x^{\ba_i} \in \C[x_1, \ldots, x_n]$, then every initial form of $p$ generates a prime ideal if and only if:\\

    \begin{enumerate}
    \item the supports of the monomials $x^{\ba_i}$ are disjoint,
    \item for any $i, j$, the greatest common divisor of the list $\{\ba_i, \ba_j\}$ is $1$.\\
    \end{enumerate}
\end{theorem}

\noindent
Theorem \ref{thm-well-poised-hypersurface} immediately implies that $\I_\ba$ is well-poised. In particular, $\I_\ba$ is generated by $\sum_{j =0}^n x_j^{a_i}Y_j$, which has disjointly supported monomials whose exponent vectors always contain a $1$. 

The Pl\"ucker ideal $I_{2, n+2}$ provides another example of a well-poised ideal.  The initial ideals coming from the tropical variety $\Trop(I_{2, n+2})$ were first studied by Speyer and Sturmfels in \cite{Speyer-Sturmfels}.  The initial ideals of points from the interior of a maximal face of $\Trop(I_{2, n+2})$ are precisely the prime binomial ideals which cut out certain toric varieties related to trivalent trees.  In particular, as a consequence of \cite[Theorem 3.4]{Speyer-Sturmfels}, it is known that the maximal faces of $\Trop(I_{2, n+2})$ are in bijection with trivalent trees $\T$ with $n+2$ leaves labelled in some way with the set $\{0, \ldots, n+1\}$.  Let $C_\T$ denote the face associated to the labelled tree $\T$. Let $E(\T)$ denote the set of edges of a trivalent tree $\T$.  Let $p_{ij} \in \Z^{E(\T)}$ denote the indicator vector of the unique path from leaf $i$ to leaf $j$ in $\T$.  We let $S_\T \subset \Z^{E(\T)}$ be the affine semigroup generated by the indicator vectors $p_{ij}$ for $0 \leq i < j \leq n+1$. The following is the main result we will need it what follows. 
    
\begin{theorem}\label{thm-pluckerinitial}
Let $\rho \in C_\T$, then the affine semigroup algebra $\C[S_\T]$ is isomorphic to the quotient algebra $\C[P_{ij} \mid 0 \leq i < j \leq n+1]/\In_\rho(I_{2, n+2})$.
\end{theorem}

Recall that a solution to the \emph{word problem} for a semigroup $S$ with generators $A_1, \ldots, A_m$ is a minimal set of relations of the form $A_{i_1}\cdots A_{i_t} = A_{j_1}\cdots A_{j_s}$ which can be used to transform an arbitrary relation to a trivial relation.  Theorem \ref{thm-pluckerinitial} can be restated as saying that the binomial initial forms of the Pl\"ucker generators of $I_{2, n+2}$ with respect to $\rho \in C_\T$ solve the word problem for $S_\T$. We show that an almost identical statement holds for the ideals $\I_\ba^\vee$ and a semigroup which is closely related to $S_\T$. First we define a map of polynomial rings which relates $I_{2, n+2}$ to the $\I_\ba^\vee$:\\

\[\Phi_\ba: \C[P_{ij} \mid 0 \leq i < j \leq n+1] \to \C[x_i, Z_{j,k} \mid 0 \leq i, j< k \leq n ]\]

$$\Phi_\ba(P_{0,i+1}) = x_i^{a_i}$$

$$\Phi_\ba(P_{j+1,k+1}) = Z_{j,k}$$\\

\noindent
The Pl\"ucker generators of $I_{2, n+2}$ map precisely to the generators of $\I_\ba^\vee$ found in Proposition \ref{prop-presentation-ba-dual} under this map.  As a consequence, the prime ideal $\Phi_\ba^*(\I_\ba^\vee)$ is the Pl\"ucker ideal $I_{2, n+2}$.  Let $S_\T(\ba) \subset \Q^{\T}$ be the affine semigroup generated by the indicator vectors $p_{ij}$ for $0 < i < j \leq n+1$, and $\frac{1}{a_i}p_{0i}$ for $0 < i \leq n+1$.  

\begin{theorem}\label{thm-ba-dual-well-poised}
    The ideal $\I_\ba^\vee$ is well-poised. 
\end{theorem}

\begin{proof}
There is an induced map $\Phi_\ba^*: \Trop(\I_\ba^\vee) \to \Trop(I_{2,n+2})$, which is always onto by \cite[Lemma 3.2.13]{Maclagan-Sturmfels}. For $\rho = (x_0, \ldots, x_n, \ldots, z_{ij}, \ldots) \in \Trop(\I_\ba^\vee)$ the pullback map on tropical varieties gives $\Phi_\ba^*(\rho) = (a_0x_0,\ldots, a_nx_n, \ldots, z_{ij}, \ldots) \in \Trop(I_{2, n+2})$, which is $1-1$.  

Now fix $\rho \in \Trop(\I_\ba^\vee)$, and suppose that $\Phi_\ba^*(\rho)$ lands in the interior of $C_\T \subset \Trop(I_{2,n+2})$.  We have $\langle \Phi_\ba(\In_{\Phi_\ba^*\rho}(I_{2, n+2})) \rangle \subset \In_\rho(\I_\ba^\vee)$. The initial forms of the generating set of $\I_\ba^\vee$ are the image of the $C_\T$-initial forms of the Pl\"ucker generators of $I_{2, n+2}$, and so are contained in the set $\Phi_\ba(\In_{\Phi_\ba^*\rho}(I_{2, n+2}))$.  Therefore, it suffices to show that these initial forms generate a prime ideal of height equal to the height of $\I_\ba^\vee$. We show this by arguing that these initial forms solve the word problem for $S_\T(\ba)$. Let $p_1\cdots p_m = q_1 \cdots q_m$ be an equation in the generators of $S_\T(\ba)$, in particular, the induced weighting $w$ of $E(\T)$ by these two words coincide. For either word, the contribution of $\frac{1}{a_i}p_{0i}$ modulo $a_i$ can be read off the edge containing the $i$-th leaf of $\T$ in its boundary. It follows that if we assume (without loss of generality) that $p_1\cdots p_m$ and $q_1 \cdots q_m$ share no generators, the contribution of each $\frac{1}{a_i}p_{0i}$ to each must be divisible by $a_i$. Now, after grouping these generators into subwords of length $a_i$, we have a word problem identical to that of $S_\T$. This is solvable by the initial forms $\In_\rho(\I_\ba^\vee)$ by construction. Finally, we observe that $\I_\ba^\vee$ is prime and homogeneous, so the fact that any $\rho \in \Trop(\I_\ba^\vee)$ has a prime initial ideal follows from this fact for a $\rho$ coming from the interior of a maximal face. 
\end{proof}

For a divisor class $(\alpha, \beta) \in \CL(\P\E)$ let $\RR_{(\alpha, \beta)}(\P\E)$ denote the graded subring given by the global sections of the positive multiples of $(\alpha, \beta)$. The following is a consequence of \cite[Theorem 3.1]{Cummings-Manon}.

\begin{corollary}
Let $\E$ be an irreducible toric vector bundle of rank $n$ on $\P^n$, and let $(\alpha, \beta)$ be a pseudo-effective divisor class on $\P\E$, then the ring $\RR_{(\alpha, \beta)}(\P\E)$ has a presentation by a well-poised ideal. 
\end{corollary}

For $\rho \subset |\Sigma_n|$ a ray we let $\Bl_\rho\P^n$ denote the corresponding toric blow-up of projective space and $\beta_\rho: \Bl_\rho\P^n \to \P^n$ be the blow-down map.  Given a toric vector bundle $\E$ over $\P^n$ with $\P\E$ a Mori dream space, it is natural to ask when the projectivization of the pullback bundle $\beta_\rho^*\E$ is also a Mori dream space.  A sufficient condition for this to occur, along with a presentation of the Cox ring of $\P\beta_\rho^*\E$ is given in \cite[Theorem 6.13]{Kaveh-Manon-tvb}.  Roughly speaking, one tests whether or not a certain point of $\Trop(L)$ derived from $\rho$ belongs to a certain polyhedral complex derived from $\Trop(\I)$.  If the ideal $\I$ presenting the Cox ring of $\P\E$ is well-poised, this subcomplex is all of $\Trop(L)$, meaning that $\P\beta_\rho^*\E$ is always a Mori dream space, regardless of the choice of ray $\rho$.  This leads us to the following corollary. 

\begin{corollary}\label{cor-blowup}
    Let $\E$ be an irreducible bundle of rank $n$ on $\P^n$, then for any ray $\rho$, $\P(\beta_\rho^*\E)$ is a Mori dream space. 
\end{corollary}

\subsection{Newton-Okounkov bodies}\label{sec-nok}

In \cite{George-Manon-Positivity} methods are described for computing global Newton-Okounkov bodies and the Newton-Okounkov bodies of divisors for a projectivized toric vector bundle $\P\E$ with $\RR(\P\E)$ generated in $\Sym$-degree $1$.  Any irreducible toric vector bundle of rank $n$ on $\P^n$ has this property by Proposition \ref{prop-presentation-ba} and \ref{prop-presentation-ba-dual}.  

The procedure in \cite[Section 3]{George-Manon-Positivity} uses the prime cone method for computing Newton-Okounkov bodies introduced in \cite{Kaveh-Manon-NOK}. We require a matrix which is built from a choice of points in a chosen prime cone.  Suppose that the toric vector bundle $\E$ comes from the pair $(L, D)$ with $L \subset \C[y_1, \ldots, y_m]$, and that $\RR(\P\E)$ is presented by the ideal $\I \subset \C[Y_1, \ldots, Y_m, x_1, \ldots, x_n]$.  Let $K \subset \Trop(L)$ be a maximal face.  The face $K$ corresponds to a full flag of flats $F_1 \subset \cdots \subset F_r = \MM(L)$. We let $E_K$ be the $r \times m$ matrix whose rows are the indicator vectors of the flats $F_i$ in the flag corresponding to $K$. We let $M$ be the following matrix:\\

\[M = \begin{bmatrix} D & -I \\ E_K & 0 \\  \end{bmatrix}. \]\\

\noindent
For a divisor class $(\alpha, \beta) \in \Z^2 \cong \CL(\P^n)\times \Z \cong \CL(\P\E)$ we let $P_{\alpha, \beta} \subset \Q_{\geq 0}^{m+n}$ be the rational polytope of points $(y_1, \ldots, y_m, x_1, \ldots, x_n)$ where $\beta = \sum_{j = 1}^m y_j$ and $\alpha = \sum_{j =1}^m y_j\bd_j - \sum_{i =1}^n x_i$. We define two polyhedra using the matrix $M$:\\

\[\Delta_C = M \circ \Q_{\geq 0}^{m+n}\]

\[\Delta_C(\alpha, \beta) = M \circ P_{\alpha, \beta}.\]\\/

\noindent
Finally, for $L$ and $\I$ as above, there is a surjection $\phi: \Trop(\I) \to \Trop(L)$:\\ 

\[\phi(v_1, \ldots, v_m, m_1, \ldots, m_n) = (\ldots, v_j + \sum_{i =1}^n m_iD_{ij}, \ldots),\]\\

\noindent
which takes faces of $\Trop(\I)$ into faces of $\Trop(L)$. The map $s: \Trop(L) \to \Trop(\I)$ given by $s(v_1, \ldots, v_m) = (v_1, \ldots, v_m, 0, \ldots, 0)$ is a section to $\phi$. Any point in $\Trop(\I)$ can be translated to a point in the image of $s$ by an element of the lineality space of $\Trop(\I)$. The following is \cite[Proposition 3.2]{George-Manon-Positivity}.

\begin{proposition}\label{prop-preimagecone}
Let $\E$ be a toric vector bundle over $\P^n$ with $\RR(\P\E)$ generated in $\Sym$-degree $1$, and let $C \subset \Trop(\I)$ be a prime cone of the form $C = \phi^{-1}K$ for a maximal face $K \subset \Trop(L)$, then $\Delta_C$ is a global Newton-Okounkov body of $\P\E$ and $\Delta_C(\alpha, \beta)$ is the Newton-Okounkov body of the divisor class $(\alpha, \beta)$. 
\end{proposition}

The assumption that the inverse image of any $K \subset \Trop(L)$ consists of prime points in $\Trop(\I)$ is satisfied by $\E_\ba$ and $\E_\ba^\vee$ due to Theorems \ref{thm-pluckerinitial} and \ref{thm-ba-dual-well-poised}.  Next we show that for any maximal face $K$
of $\Trop(L_n)$ (respectively $\Trop(L_n^\ba)$), the inverse image lies in a face of $\Trop(\I_\ba)$ (respectively $\Trop(\I_\ba^\vee)$). Then we describe the indicator matrices $E_K$ for certain faces of $\Trop(L_n)$ and $\Trop(L_n^\vee)$. 

\begin{proposition}
    For the bundle $\E_\ba$, and any flag of subsets $A_1 \subset \cdots \subset A_{n-1} \subset A_n = \{0, \ldots, n\}$ with $|A_i| = i$ for $i < n$, we can take the matrix $E_K$ to have rows the indicator vectors of the $A_i$.  
\end{proposition}

\begin{proof}
First, any collection $A \subset \{0, \ldots, n\}$ with $|A| < n$ is a flat of $\MM(L_n)$.
Let $\be_A$ denote the indicator vector of the set $A$.  A point in $u \in K$ is positive weighted combination $u = \sum_{i =1}^n v_i \be_{A_i}$.  The image $s(u) \in \Trop(\I_\ba)$ does not affect the $x_i$ variables, and weights precisely two $Y_j$ equal and less than the others.  It follows that the initial form of $\sum_{j =0}^n x_j^{a_j}Y_j$ is the same for any point of $K$.
\end{proof}

The case $L_n^\vee$ is more involved. Our argument works by interpreting $s(u) \in \Trop(\I_\ba^\vee)$ for $u \in K \subset \Trop(L_n^\vee)$ as a weighted trivalent tree.  The latter determines a point of the tropical variety of the Grassmannian $\Gr_2(n+2)$ by \cite{Speyer-Sturmfels}, and therefore a point of $\Trop(\I_\ba^\vee)$ by Theorem \ref{thm-ba-dual-well-poised}.  In particular, we show that the tree type of $s(u)$ is the same for all $u$ taken from the interior of $K$.  To simplify matters we deal with flats which are spanned by subsets of the basis $\{z_{01}, \ldots, z_{n-1,n}\}$.  Other cases are related to this case by the action of the permutation group. Let a set $F \subset \{z_{ij} \mid i < j\}$ have the property that for indices $k < \ell$, if $\{z_{k,k+1}, \ldots, z_{\ell-1,\ell}\} \subset F$ then $z_{k,\ell} \in F$.  

\begin{proposition}
    For the bundle $\E_\ba^\vee$, and any flag of subsets $F_1 \subset \cdots \subset F_{n-1} \subset F_n = \{z_{ij} \mid i < j\}$ we can take the matrix $E_K$ to have rows the indicator vectors of the $F_i$.
\end{proposition}

\begin{proof}
    Each set $F$ with the above property is a flat of $\MM(L_n^\vee)$.

    A point $u \in K$ is a weighted combination $\sum_{i = 1}^n v_i\be_{F_i}$.
    The first flat $F_1$ is of the form $\{z_{i,i+1}$.  We start by building the tree $\T_1$ which has leaves labelled $i+1, i+2$, along with a new leaf $w''$, all tied to an internal vertex $w'$.  We label the edge between $w'$ and $w''$ with $v_1$.  
    
    Now suppose the tree $\T_j$ has been built, and that its leaves are labelled by the indices appearing in the elements of $F_k$ shifted up by $1$, along with a ``root" leaf $w$.  Moving to $F_{k+1}$ introduces a single new element of the form $z_{j,j+1}$. If both indices do not appear in $F_k$, we add in new leaves labelled $j+1, j+2$, a new internal vertex $w'$, and connect $w'$ to both indexed leaves and the root $w$. We the label the edge between $w$ and $w'$ with $v_{k+1}$.  If $z_{j,j+1}$ shares an index with $F_k$, we add in a single new leaf labelled by the new index, a new root $w'$, and connect the old root $w$ to the new root $w'$. The edge between the old and new root is then labelled $v_{k+1}$.  The last root is given the label $0$. Keeping with the definition of the section function, we then choose the weights on the leaf edges to be negative numbers so that the total weight of any path from $0$ to an index $i$ is $0$.  We then divide every weight on the resulting tree $\T$ by $2$. 
    
    We compute a point on $\Trop(\I_\ba^\vee)$ by sending $x_i$ to the negative of the total weight of those edges in the unique path from $0$ to the vertex $i+1$, and $Z_{ij}$ to the negative of the total weight of those edges in the unique path from $i+1$ to $j+1$.  In the second case, this is the sum of those $v_k$ corresponding to the flats $F_k$ which contain $z_{ij}$. In the first case, it is always $0$. As a consequence, the tree $\T$ produces the point $s(u) \in \Trop(\I_\ba^\vee)$. By construction, the topology of $\T$ only depends on the chosen flag of flats.    
\end{proof}

\begin{example} 
   Let $\ba = \{a_0, a_1, a_2, a_3\}$ be positive integers. We construct Newton-Okounkov bodies for
   the projectivization of the irreducible bundle $\E_\ba^\vee$. For the maximal flag of flats we choose $F_0 = \langle z_{01} \rangle$, $F_1 = \langle z_{01},z_{12} \rangle$, and $F_2 = \MM(L_n^\vee)$. The indicator vectors of this flag make the bottom three rows of the matrix $M$:

   \[  \begin{matrix}
        z_{01} & z_{02} & z_{03} & z_{12} & z_{13} & z_{23} & x_0 & x_1 & x_2 & x_3\\
        \hline\\
       
        0 & 0 & 0 & a_0 & a_0 & a_0 & -1 & 0 & 0 & 0\\
        0 & a_1 & a_1 & 0 & 0 & a_1 & 0 & -1 & 0 & 0\\
        a_2 & 0 & a_2 & 0 & a_2 & 0 & 0 & 0 & -1 & 0\\
        a_3 & a_3 & 0 & a_3 & 0 & 0 & 0 & 0 & 0 & -1\\
        1 & 1 & 1 & 1 & 1 & 1 & 0 & 0 & 0 & 0\\
        1 & 1 & 0 & 1 & 0 & 0 & 0 & 0 & 0 & 0\\
        1 & 0 & 0 & 0 & 0 & 0 & 0 & 0 & 0 & 0
   \end{matrix}\]\\

    \noindent
    The global Newton-Okounkov body associated to this choice of flag is $M\circ \Q_{\geq 0}^{6 + 4}$.  The Newton-Okounkov body of a divisor $(\alpha, \beta) \in \Z^2 \cong \CL(\P\E_\ba^\vee)$ is the image of the polytope $P_{\alpha, \beta} \subset \Q_{\geq 0}^{6+4}$ given by those tuples $(z_{01}, \ldots, z_{23}, x_0, \ldots, x_3)$ satisfying:\\

    \[ z_{01} + z_{02} + z_{03} + z_{12} + z_{13} + z_{23} = \beta\]
    \[ (a_2 + a_3)z_{01} + (a_1 + a_3)z_{02} + (a_1 + a_2)z_{03} + (a_0 + a_3)z_{12} + (a_0 +a_2)z_{13} + (a_0 + a_1)z_{23} - x_0 - x_1 -x_2 - x_3 = \alpha  \]\\

    \noindent
    under the matrix $M$.

    The flag $F_0 \subset F_1 \subset F_2 = \MM(L_n^\vee)$ corresponds to the tree in Figure \ref{fig-tree}. We label edges of $\T$ by the corresponding flat of $\MM(L_n^\vee)$. \\

\begin{figure}[ht]
\begin{tikzpicture}
\node[shape=circle,draw=black] (1) at (0,0) {1};
\node[shape=circle,draw=black] (2) at (4,0) {2};
\node[shape=circle,draw=black] (3) at (8,0) {3};
\node[shape=circle,draw=black] (4) at (12,0) {4};
\node[shape=circle,draw=black] (W1) at (2,-3) {$u$};
\node[shape=circle,draw=black] (W2) at (6,-3) {$v$};
\node[shape=circle,draw=black] (W3) at (10,-3) {$w$};
\node[shape=circle,draw=black] (0) at (8,-6) {0};

\path [-] (1) edge node[left] {} (W1);
\path [-] (2) edge node[left] {} (W1);
\path [-] (W1) edge node[above] {$F_0$} (W2);
\path [-] (3) edge node[left] {} (W2);
\path [-] (W2) edge node[above] {$F_1$} (W3);
\path [-] (4) edge node[left] {} (W3);
\path [-] (W3) edge node[right] {$F_2$} (0);
\end{tikzpicture}
\caption{The tree $\T$ corresponding to the flag $F_0 \subset F_1 \subset F_2$.}
\label{fig-tree}
\end{figure}
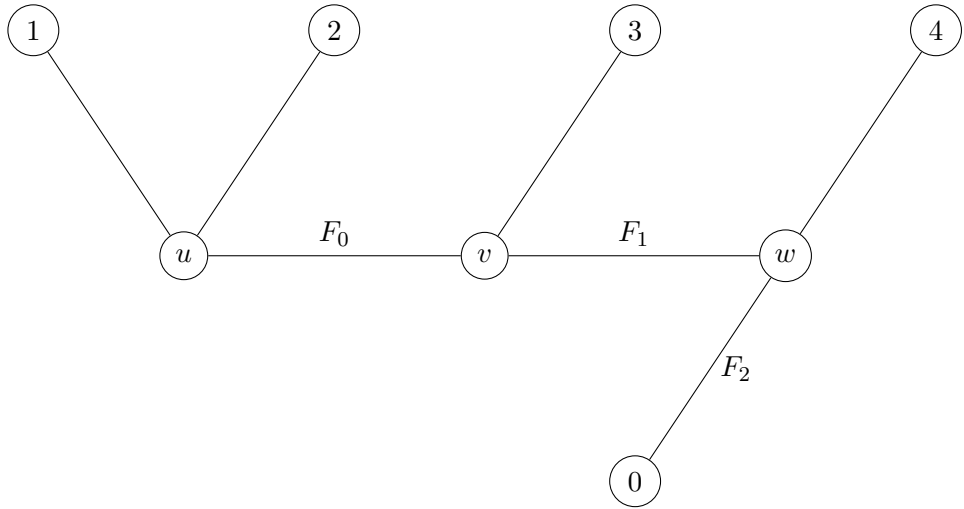

\noindent
When $\ba = \{1, 1, 1, 1\}$, the Cox ring $\RR(\P\E_\ba^\vee)$ is isomorphic to the Pl\"ucker algebra of the Grassmannian $\Gr_2(5)$.  The tree $\T$ above is recognizable as a \emph{caterpillar} tree, and the resulting Newton-Okounkov cone for $\Gr_2(5)$ is derived from the Gel'fand-Zetlin patterns with $2$ columns and $5$ rows. See \cite[Section 14.4]{Miller-Sturmfels} for an account of the Gel'fand-Zetlin degeneration of a flag variety.  
\end{example}

\section{Positivity properties of divisors}\label{sec-positivity}

In this section we compute the monoid $\Bpf(\P\E) \subset \CL(\P\E)$ of basepoint-free divisor classes for $\E$ an irreducible toric vector bundle of rank $n$ on $\P^n$.  We then prove that any Nef class is basepoint-free, and any ample class is very ample on $\P\E$.  The proof uses the fact that the Cox ring $\RR(\P\E)$ is always generated in $\Sym$-degree $1$; this means that $\RR(\P\E)$ is generated by the section spaces of the form $
H^0(\P^n, \O(d)\otimes \E)$. Generation in $\Sym$-degree $1$ implies that $\Bpf(\P\E)$ has an expression in terms of certain matroids associated to $\E$.  

\subsection{The basepoint-free monoid}

For the following see \cite[Proposition 3.3.2.6]{ADUH-book}. Let $X$ be a Mori dream space with Cox ring $\RR(X)$ generated by $f_1, \ldots, f_m$.  For any point $p \in X$ there is a monoid $S_p \subset \CL(X)$ consisting of those divisor classes which carry a section which does not vanish at $p$.  Clearly $\Bpf(X) = \bigcap_{p \in X} S_p$, and it is straightforward to show that each $S_p$ is generated by the degrees of the $f_i$ with $f_i(p) \neq 0$.  As a consequence, there are only a finite number of possible distinct $S_p$, however one still needs to find a set of representative points $p$.  

Now suppose we have a toric vector bundle $\E$ corresponding to a pair $(L, D)$ over a smooth, projective toric variety $X(\Sigma)$.  For each facet $\sigma \in \Sigma$ there is an initial linear ideal $\In_\sigma(L)$ (\cite[Section 2.1]{George-Manon-Positivity}), and an initial matroid $\MM(\In_\sigma(L))$ of $\MM(L)$.  Any maximal, non-trivial flat $F \subset \MM(\In_\sigma(L))$ defines a point $p_{\sigma, F} \in \P\E_\sigma$. The following is \cite[Proposition 4.1]{George-Manon-Positivity}.

\begin{theorem}\label{thm-bpf}
Let $\E$ and $X(\Sigma)$ be as above, and suppose that $\RR(\P\E)$ is generated in $\Sym$-degree $1$, then $\Bpf(\P\E) = \bigcap_{\sigma, F} S_{p_{\sigma, F}}$. 
\end{theorem}

If $\RR(\P\E)$ is generated in $\Sym$-degree $1$, the effective monoid is spanned by the class-group degrees of the generators $x_1, \ldots, x_n$ and $Y_1, \ldots, Y_m$.  Over $\P^n$, these degrees are $\deg(x_i) = (-1, 0)$ and $\deg(Y_j) = (\bd_j, 1)$, where $\bd_j = \sum_{i = 0}^n D_{ij}$.

We say the bundle $\E$ is a \emph{monomial bundle} \cite[Section 4.2]{George-Manon-Positivity} if each initial ideal $\In_\sigma(L)$ is a monomial ideal. Equivalently, $(L, D)$ defines a monomial ideal when, for any facet $\sigma \in \Sigma$, the minimal face of the Gr\"obner fan of $L$ containing the rows of $D$ corresponding to the rays $\sigma(1)$ is maximal. In this case, the flats $F \subset \MM(\In_\sigma (L))$ are complements of single elements.  If $\E$ is a bundle over $\P^n$ and the monoids $S_{p_{\sigma, F}}$ of a monomial bundle are generated by the classes $(-1, 0)$ and $(\bd_j, 1)$ for one of the generators $Y_j \in \RR(\P\E)$. In particular, any $S_{p_{\sigma, F}}$ in this case is freely generated. 
 
\begin{lemma}\label{lem-monomial}
    Let $\E$ be an irreducible toric vector bundle of rank $n$ over $\P^n$, then $\E$ is  monomial. In particular, for a facet $\sigma_i$ in the fan of $\P^n$,\\ 
    \begin{enumerate}
        \item $\In_{\sigma_i}(L_n) = \langle y_i \rangle \subset \C[y_0, \ldots, y_n]$,
        \item $\In_{\sigma_i}(L_n^\vee) = \langle z_{jk} \mid i \notin \{j, k\} \rangle \subset \C[z_{ij} \mid 0 \leq i < j \leq n]$.\\
    \end{enumerate}
\end{lemma}

\begin{proof}
For both classes of ideal we consider the face $\sigma_0$ in the fan of $\P^n$ spanned by the elementary basis vectors. The ideal $L_n$ is generated by the form $y_0 + \cdots + y_n$.  The rows of $D_\ba$ corresponding to the elements of $\sigma_0(1)$ are each of the form $(0, \ldots, a_i, \ldots, 0)$ where $i \neq 0$.  The initial ideal of the face of the Gr\"obner fan of $L_n$ must be generated by a common refinement of the initial forms of these rows, and the only variable not given a positive weight by some row is $y_0$. This shows $\In_{\sigma_0}(L_n) = \langle y_0 \rangle$.

For $L_n^\vee$ we use the non-negative diagram $D_\ba'$.  Recall that the $i$-th row of this diagram weights $z_{jk}$ with $a_i$ if $i \notin \{j,k\}$ and $0$ otherwise.  We show that $\In_{\sigma_0}(L_n^\vee)$ must containing $z_{jk}$ for any $0 < j < k$. This proves the lemma for dimension reasons. The initial ideal of the minimal face of the Gr\"obner containing the rows from $\sigma_0(1)$ must itself be an initial ideal of the sum of any two rows.  Consider the initial ideal of the sum of the $j$ and $k$-th rows; this weights $z_{0k}$ with $a_j$, $z_{0j}$ with $a_k$ and $z_{jk}$ with $0$.  It follows that the initial form of $z_{0k} - z_{0j} - z_{jk}$ with respect to this sum is $z_{jk}$, and that $z_{jk} \in \In_{\sigma_0}(L_n^\vee)$ for all $0 < j < k$.
\end{proof}

\subsection{Embeddings into toric varieties}

Observe that Lemma \ref{lem-monomial} implies that for any irreducible toric vector bundle of rank $n$ on $\P^n$, $\Bpf(\P\E)$ is the intersection of saturated monoids, and is therefore saturated. This implies that any Nef class of $\P\E$ is basepoint free.  In this section we recover this fact and more by showing that any $\P\E$ comes with a useful embedding into a smooth, projective toric variety. 

Let $\phi:X \to Z$ be an embedding of a Mori dream space into a toric variety $Z$. This embedding is said to be \emph{neat} (\cite[Definition 3.2.5.2]{ADUH-book}) 
if it induces an isomorphism $\CL(X) \cong \CL(Z)$.  We say that $\phi$ is \emph{neat and tidy} it also induces $\Bpf(X) \cong \Bpf(Z)$. 

In \cite[Proposition 2.6]{George-Manon-Positivity} it is shown that if $\E$ is a toric vector bundle corresponding to the pair $(L, D)$ with $\RR(\P\E)$ generated in $\Sym$-degree $1$ then $\P\E$ has a neat embedding into the projectivization of a split toric vector bundle.  In particular, if the $Y_j \in \RR(\P\E)$ corresponding to the variables $y_j$ in the polynomial ring containing $L$ generate $\RR(\P\E)$, one can use the split toric vector bundle $\V_D = \bigoplus_{j =1}^m \O(D_j)$, where $\O(D_j)$ is the $T$-linearized line bundle corresponding to the $j$-th column of the diagram $D$ for the embedding. 

\begin{proposition}\label{prop-neatandtidy}
    Let $\E$ be an irreducible toric vector bundle of rank $n$ on $\P^n$, then $\P\E$ has a neat and tidy embedding into the projectivization of a split toric vector bundle.  As a consequence, any Nef class of $\P\E$ is basepoint free, and any ample class of $\P\E$ is very ample.  
\end{proposition}

\begin{proof}
We consider the pairs $(L_n, D_\ba)$ and $(L_n^\vee, D_\ba')$.  In both cases, the basepoint-free monoid of the corresponding bundle $\P\V_D$ is the intersection of the monoids $\Z_{\geq 0}\{(-1,0), (\bd_j, 1)\}$, where $\bd_j$ runs over all the classes corresponding to the columns of the diagram $D$. To prove the first statement, it suffices to show that the basepoint-free monoids of $\P\E_\ba$ and $\P\E_\ba^\vee$ have the same description.  This translates to showing that for any $y_i$ (respectively  $z_{jk}$) there is a face $\sigma_\ell$ of the fan of $\P^n$ for which $y_i$ (respectively $z_{jk}$) is not in the initial ideal $\In_{\sigma_\ell}(L_n)$ (respectively $\In_{\sigma_\ell}(L_n^\vee)$).  But this is a consequence of Lemma \ref{lem-monomial}.  

As a consequence, any ample class of $\P\E$ is the restriction of any ample class from $\P\V_D$. Any ample line bundle on a smooth toric variety is very ample, hence any ample class on $\P\E$ is very ample. 
\end{proof}

\bibliographystyle{alpha}
\bibliography{main}

\end{document}